\documentclass[a4paper,12pt]{amsart}
\usepackage{amsmath,amsfonts,amssymb}
\usepackage{mathabx}
\usepackage{graphicx}
\usepackage{url}
\usepackage{hyperref}
\usepackage{comment}
\usepackage{multirow}

\calclayout

\usepackage{graphicx}
\allowdisplaybreaks

\theoremstyle{plain}
\newtheorem{theorem}{Theorem}[section]
\newtheorem{proposition}[theorem]{Proposition}
\newtheorem{lemma}[theorem]{Lemma}

\theoremstyle{definition}
\newtheorem*{definition}{Definition}

\theoremstyle{remark}

\newcommand{\Z}{\mathbb{Z}}
\newcommand{\Gc}{\mathcal{G}}

\begin{document}

\title[Remarks on the Liechti-Strenner's examples having small dilatations]{Remarks on the Liechti-Strenner's examples having small dilatations}

\author{Ji-Young Ham and Joongul Lee}
\address{Department of Science, Hongik  University, 
94 Wausan-ro, Mapo-gu, Seoul,
04066,\\
  Korea} 
\email{jyham@hongik.ac.kr}

\address{Department of Mathematics Education, Hongik University, 
94 Wausan-ro, Mapo-gu, Seoul,
04066\\
   Korea} 
\email{jglee@hongik.ac.kr}



\begin{abstract}
We show that the Liechti-Strenner's example for the closed nonorientable surface in~\cite{LiechtiStrenner18} minimizes the dilatation within the class of pseudo-Anosov homeomorphisms with an orientable invariant foliation and all but the first coefficient of the characteristic polynomial of the action induced on the first cohomology nonpositive. We also show that the Liechti-Strenner's example of orientation-reversing homeomorphism for the closed orientable surface in~\cite{LiechtiStrenner18} minimizes the dilatation within the class of pseudo-Anosov homeomorphisms with an orientable invariant foliation and all but the first coefficient of the characteristic polynomial $p(x)$ of the action induced on the first cohomology nonpositive or all but the first coefficient of $p(x) (x \pm 1)^2$, $p(x) (x^2 \pm 1)$, or $p(x) (x^2 \pm x + 1)$ nonpositive.
\end{abstract}

\maketitle
\section{Introduction} 

Let $\Sigma_g$ be a closed surface of finite genus.
A homeomorphism $h$ of $\Sigma_g$ is called \emph{pseudo-Anosov} if there is a pair of transversely measured foliations 
$\mathcal{F}^u$ and $\mathcal{F}^s$ in $\Sigma$ and a real number $\lambda >1$ such that $h(\mathcal{F}^u)=\lambda \mathcal{F}^u$ and  $h(\mathcal{F}^s)=1/\lambda \mathcal{F}^s$~\cite{Thurston88,CassonBleiler88}.
The number $\lambda$ is called the \emph{dilatation} of $h$ and the logarithm of $\lambda$ is called the \emph{topological entropy}. The set of dilatations of pseudo-Anosov homeomorphisms of the group of isotopy classes $\Sigma_g$ is discrete~\cite{ArnouxYoccoz81, Ivanov88}. In particular, there exists the minimal dilatation.

The dilatation of a pseudo-Anosov homeomorphism of $\Sigma_g$ measures its dynamical complexity. 
Furthermore, the collection of topological entropies has a geometric interpretation as the collection of Teichm\"uller distances
between Riemann surfaces of the same topological type as $\Sigma_g$~\cite{Abikoff80} if $\Sigma_g$ is orientable (as a subcollection of Teichmu\"ller distances between those of the same topological type as the orientable double cover of $\Sigma_g$ if $\Sigma_g$ is nonorientable).
In particular, the logarithm of the minimal dilatation of a genus $g$ surface gives the length of the systole for the genus $g$ moduli space if $\Sigma_g$ is orientable.

For an orientable surface $S_g$, several results have been known on the bounds of the minimal dilatation, $\delta_g$, for all pseudo-Anosov homeomorphisms of 
$S_g$.
Penner gave upper and lower bounds for the dilatations of $S_g$, and
proved that as $g$ tends to infinity, the minimal dilatation tends to one (the logarithm of the minimal dilatation tends to zero on the order of 
$1/g$)~\cite{Penner91}.
The upper bound was improved by Bauer for closed surfaces of genus $g\geq3$~\cite{Bauer92}.

However, the exact value of the minimal dilatation $\delta_g$ of $S_g$ has been
found only when the genus $g$ is two~\cite{ChoHam08}.

More is known for the minimal dilatation of orientation-preserving pseudo-Anosov homeomorphisms on $S_g$ with orientable invariant foliations. Denote the minimal dilatation of orientation-preserving pseudo-Anosov homeomorphisms on $S_g$ with orientable invariant foliations by $\delta^{+}(S_g)$. The following Table~\ref{table:deltaplus} shows the known values.

\begin{table} 
\begin{tabular}{|c||c||l|} \hline
$g$ &  $\delta^{+}(S_g) \approx$ & Minimal polynomial of $\delta^{+}(S_g)$\\\hline
1 & $2.61803$ & $x^2-3 x+1$ \\
2 & $1.72208$ & $x^4 - x^3 - x^2 - x + 1$ \\
3 &$1.40127$ & $x^6-x^4-x^3-x^2+1$ \\
4 &$1.28064$ & $x^8-x^5-x^4-x^3+1$ \\
5 &$1.17628$ & $x^{10}+x^9-x^7-x^6-x^5-x^4-x^3+x+1=\frac{x^{12}-x^7-x^6-x^5+1}{x^2- x+1}$ \\
7 & $1.11548$ & $x^{14}+x^{13}-x^9-x^8-x^7-x^6-x^5+x+1$ \\
8 & $1.12876$ & $x^{16}-x^9-x^8-x^7+1$ \\
\hline
\end{tabular}
\caption{The known values of $\delta^{+}(S_g)$}
\label{table:deltaplus}
\end{table}
The pseudo-Anosov homeomorphisms realizing $\delta^{+}(S_g)$ in Table~\ref{table:deltaplus} were constructed by 
Zhirov~\cite{Zhirov95} for $g=2$,
Lanneau and Thiffeault~\cite{LanneauThiffeault11b} for $g=3$ and $4$, 
Leiniger~\cite{Leininger04} for $G=5$, 
Kin and Takasawa~\cite{KinTakasawa13} and Aaber and Dunfield~\cite{AaberDunfield10} for $g=7$, and
Hironaka~\cite{Hironaka10} for $g=8$. Hironaka~\cite{Hironaka10} then showed that all of the examples above except the g = 7 example arise from the fibration of the link complement of $6_2^2$.
 From genus $6$ to genus $8$, each example identified the lower bound calculated by Lanneau and Thiffeault~\cite{LanneauThiffeault11b} as the minimal dilatation. 

Recently, Liechti and Strenner~\cite{LiechtiStrenner18} determined the minimal dilatation of pseudo-Anosov homeomorphisms with orientable invariant foliations on the closed nonorientable surfaces of genus 4, 5, 6, 7, 8, 10, 12, 14, 16, 18 and 20 and the minimal dilatation of orientation-reversing pseudo-Anosov homeomorphisms with orientable invariant foliations on the closed orientable surfaces of genus 1, 3, 5, 7, 9, and 11. 
 Denote by $N_g$ the closed nonorientable surface of genus $g$ and by 
$\delta^{+} (N_g)$ the minimal dilatation among pseudo-Anosov homeomorphisms of $N_g$ with an orientable invariant foliation.
Denote the minimal dilatation among orientation-reversing pseudo-Anosov homeomorphisms on $S_g$ with orientable invariant foliations by $\delta^{+}_{rev}(S_g)$.  The values determined by Liechti and Strenner~\cite{LiechtiStrenner18} are in  Table~\ref{table:nonorientable} and Table~\ref{table:reversing}.

\begin{table} 
\begin{tabular}{|c||c||l|} \hline
$g$ &  $\delta^{+}(N_g) \approx$ & Minimal polynomial of $\delta^{+}(N_g)$\\\hline
4 & $1.83929$ & $x^3-x^2- x-1$ \\
5 & $1.51288$ & $x^4 - x^3 - x^2 + x - 1$ \\
6 &$1.42911$ & $x^5-x^3-x^2-1$ \\
7 &$1.42198$ & $x^6-x^5-x^4-x^3+x-1$ \\
8 &$1.28845$ & $x^7-x^4-x^3-1$ \\
10 & $1.21728$ & $x^{11}-x^6-x^5-1$ \\
12 & $1.17429$ & $x^{13}-x^7-x^6-1$ \\
14 & $1.14551$ & $x^{15}-x^8-x^7-1$ \\
16 & $1.12488$ & $x^{16}-x^9-x^8-x^7+1$ \\
18 & $1.10938$ & $x^{17}-x^9-x^8-1$ \\
20 & $1.09730$ & $x^{19}-x^{10}-x^9-1$ \\
\hline
\end{tabular}
\caption{The known values of $\delta^{+}(N_g)$}
\label{table:nonorientable}
\end{table}

\begin{table} 
\begin{tabular}{|c||c||l|} \hline
$g$ &  $\delta^{+}_{rev}(S_g) \approx$ & Minimal polynomial of $\delta^{+}_{rev}(S_g)$\\\hline
1 & $1.61803$ & $x^2- x-1$ \\
3 & $1.25207$ & $\frac{x^8 - x^5 - x^3- 1}{x^2+1}$ \\
5 &$1.15973$ & $\frac{x^{12} - x^7 - x^5- 1}{x^2+1}$ \\
7 &$1.11707$ & $\frac{x^{16} - x^9 - x^7- 1}{x^2+1}$ \\
9 &$1.09244$ & $\frac{x^{20} - x^{11} - x^9- 1}{x^2+1}$ \\
11 & $1.07638$ & $x^{24}-x^{13}-x^{11}-1$ \\
\hline
\end{tabular}
\caption{The known values of $\delta^{+}_{rev}(S_g)$}
\label{table:reversing}
\end{table}

The main purpose of the paper is to show that the Liechti-Strenner's example for the closed nonorientable surface in~\cite{LiechtiStrenner18} minimizes the dilatation within the class of pseudo-Anosov homeomorphisms with an orientable invariant foliation and all but the first coefficient of the characteristic polynomial of the action induced on the first cohomology nonpositive. We also show that the Liechti-Strenner's example of orientation-reversing homeomorphism for the closed orientable surface in~\cite{LiechtiStrenner18} minimizes the dilatation within the class of pseudo-Anosov homeomorphisms with an orientable invariant foliation and all but the first coefficient of the characteristic polynomial $p(x)$ of the action induced on the first cohomology nonpositive or all but the first coefficient of $p(x) (x \pm 1)^2$, $p(x) (x^2 \pm 1)$, or $p(x) (x^2 \pm x + 1)$ nonpositive.
\begin{theorem} \label{thm:N}
Denote by $N_{2k}$ the closed nonorientable surface of genus $2k$. 
For all $k \geq 2$, the largest positive root of
$$ x^{2k-1}-x^k-x^{k-1}-1$$ minimizes the dilatation within the class of pseudo-Anosov homeomorphisms with an orientable invariant foliation and all but the first coefficient of the characteristic polynomial of the action induced on the first cohomology nonpositive.
\end{theorem}

\begin{theorem} \label{thm:S}
 Denote by $S_{2k-1}$ the closed orientable surface of genus $2k-1$.
For all $k \geq 2$, the largest positive root of
$$ x^{4k}-x^{2k+1}-x^{2k-1}-1$$ minimizes the dilatation within the class of pseudo-Anosov homeomorphisms with an orientable invariant foliation and all but the first coefficient of the characteristic polynomial $p(x)$ of the action induced on the first cohomology nonpositive or all but the first coefficient of $p(x) (x \pm 1)^2$, $p(x) (x^2 \pm 1)$, or $p(x) (x^2 \pm x + 1)$ nonpositive.
\end{theorem}
\section{Perron-Frobenius matrix} \label{section:perron}

\begin{definition}
Let $M=\begin{pmatrix}m_{ij}\end{pmatrix}$ and $N= \begin{pmatrix}n_{ij}\end{pmatrix}$ be two nonnegative $d \times d$ matrices. We say $M > N$ 
if  $m_{ij} \geq n_{ij}$ for all $i,\, j \in \{1,2,\ldots,d\}$ and strict inequality holds for at least one entry.
\end{definition}

\begin{lemma} [Perron-Frobenius] \label{lemma:Perron}
Let $T$ and $L$ be two Perron-Frobenius matrices such that $T > L$, then the spectral radius $\lambda_T$ of $T$ is strictly bigger than the spectral radius 
$\lambda_L$ of $L$.
\end{lemma}

\begin{proof}
Let $t > 0$ be a left eigenvector of $T$ corresponding to $\lambda_T$ and let $l > 0$ be a right eigenvector of $L$ corresponding to $\lambda_L$. Then
$$Tl > Ll = \lambda_Ll$$
and
$$\lambda_Ttl=tTl>\lambda_Ltl.$$
Therfore $\lambda_T>\lambda_L$.
\end{proof}

\begin{proposition}
The transpose of the Frobenius companion matrix of $$x^{2k-1}-a_{2k-2}x^{2k-2} \cdots -a_{1}x-1$$
 \noindent is
$$\begin{bmatrix}
0 & 1 & 0 & \hdots & 0 \\
0 & 0 & 1 & \hdots & 0 \\
\vdots & \vdots & \vdots & \ddots & \vdots \\
0& 0 & 0 & \hdots 1 \\
1 & a_1 & a_2 & \hdots & a_{2k-2}
\end{bmatrix},$$
and the transpose of the Frobenius companion matrix of $$x^{4k}-a_{4k-1}x^{4k-1} \cdots -a_{1}x-1$$
\noindent is
$$\begin{bmatrix}
0 & 1 & 0 & \hdots & 0 \\
0 & 0 & 1 & \hdots & 0 \\
\vdots & \vdots & \vdots & \ddots & \vdots \\
0& 0 & 0 & \hdots 1 \\
1 & a_1 & a_2 & \hdots & a_{4k-1}
\end{bmatrix}.$$
\end{proposition}

\begin{proposition} \cite[Proposition 4.1]{LiechtiStrenner18} \label{pro:LS1}
Let $\psi : N_g \rightarrow N_g$ be a pseudo-Anosov map with a transversely orientable invariant foliation on the closed nonorientable surface $N_g$ of genus $g$. Then its dilatation $\lambda$ is a root of a (not necessarily irreducible) polynomial $p(x) \in \Z[x]$ with the following properties:
\begin{enumerate}
\item $\textnormal{deg}(p)=g-1$
\item $p(x)$ is monic and its constant coefficient is $\pm1$
\item The absolute values of the roots of $p(x)$ other than $\lambda$ lie in the open interval 
$(\lambda^{-1},\lambda)$. In particular, $p(x)$ is not reciprocal or anti-reciprocal.
\item $p(x)$ is reciprocal mod $2$.
\end{enumerate}
\end{proposition}

\begin{proposition} \cite[Proposition 4.3]{LiechtiStrenner18} \label{pro:LS2}
Let $\psi : S_g \rightarrow S_g$ be an orientation-reversing pseudo-Anosov map with  transversely orientable invariant foliations. Then its dilatation 
$\lambda$ is a root of a (not necessarily irreducible) polynomial $p(x) \in \Z[x]$ with the following properties:
\begin{enumerate}
\item $\textnormal{deg}(p)=2g$
\item $p(x)$ is monic and its constant coefficient is $(-1)^g$
\item $p(x)=(-1)^g x^{2g} p(-x^{-1})$
\item The absolute values of the roots of $p(x)$ other than $\lambda$ and $-\lambda^{-1}$ lie in the open interval $(\lambda^{-1},\lambda)$.
\end{enumerate}
\end{proposition}

\begin{lemma} \label{lemma:DM}
Let $A$ be the adjacency matrix for a graph $G$ with $m$ vertices.
Then $A$ is primitive if and only if G is strongly connected and the gcd of the lengths of the loops in $G(A)$ is one.
\end{lemma}

\begin{proof}
See~\cite[Chapter 6, Section 1, Remarks, 6, 8]{DulmageMendelsohn67}.
\end{proof}

\begin{lemma} \label{lemma:mod2}
\begin{enumerate}
\item Let $\psi : N_{2k} \rightarrow N_{2k}$ be a pseudo-Anosov map with a transversely orientable invariant foliation on the closed nonorientable surface $N_{2k}$ of genus $2k$. Let $x^{2k-1}-a_{2k-2}x^{2k-2} \cdots -a_{1}x-1$ be the characteristic polynomial of the action of $\psi$ induced on first cohomology of $N_{2k}$, 
then $a_{i} \equiv a_{2k-1-i}$ mod $2$ and at least for one $a_{i}$ with $gcd(2k-1,2k-1-i)=1$, $a_{i} \neq 0$ for $1 \leq i \leq k-1$.
\item Let $\psi : S_{2k-1} \rightarrow S_{2k-1}$ be an orientation-reversing pseudo-Anosov map with  transversely orientable invariant foliations.  If $p(x)=x^{4k-2}-a_{4k-3}x^{4k-3} \cdots -a_{1}x-1$ be the characteristic polynomial of the action of $\psi$ induced on first cohomology of $S_{2k-1}$, 
then $a_{i} \equiv a_{4 k-2-i}$ mod $2$ and at least for one $a_{i}$ with $gcd(4k-2,4k-2-i)=1$, $a_{i} \neq 0$ for $1 \leq i \leq 2 k-1$.
\item Let $\psi : S_{2k-1} \rightarrow S_{2k-1}$ be an orientation-reversing pseudo-Anosov map with  transversely orientable invariant foliations. Let $p(x)$ be the characteristic polynomial of the action of $\psi$ induced on first cohomology of $S_{2k-1}$. If $ x^{4k}-a_{4k-1}x^{4k-1} \cdots -a_{1}x-1$ is one of $p(x) (x \pm 1)^2$, $p(x) (x^2 \pm 1)$, or $p(x) (x^2 \pm x + 1)$, then $a_{i} \equiv a_{4 k-i}$ mod $2$ and at least for one $a_{i}$ with $gcd(4k,4k-i)=1$, $a_{i} \neq 0$ for $1 \leq i \leq 2 k-1$.
\end{enumerate}
\end{lemma}

\begin{proof}
\begin{enumerate}
\item[] 
\item[] (1) can be obtained from Proposition~\ref{pro:LS1} (1) and (4) and Lemma~\ref{lemma:DM}.

\item[] (2) can be obtained from Proposition~\ref{pro:LS2} (1) and (3) and Lemma~\ref{lemma:DM}.

\item[] Note that $p(x) (x \pm 1)^2$, $p(x) (x^2 \pm 1)$, or $p(x) (x^2 \pm x + 1)$ is either $p(x) (x^2 + 1)$, or $p(x) (x^2 + x + 1)$ mod $2$. Hence,
\item[] (3) can be obtained from (2) and Lemma~\ref{lemma:DM}.
\end{enumerate}
\end{proof}

By a spectral radius of a polynomial, we mean the spectral radius of the Frobenius companion matrix of it.

\begin{lemma} \label{lemma:minispec}
Let $n \geq 2$.
$$ x^{2n}-x^{n+1}-x^{n-1}-1$$ gives the minimum spectral radius among the polynomials $x^{2n}-a_{2n-1}x^{2n-1} \cdots -a_{1}x-1$ with $a_{i} \equiv a_{2 n-i}=1$
for only one $i$ ($1 \leq i \leq  n-1$) and $a_{j}=0$ for all $j$ with $j \neq i$ and $j \neq 2 n-i$.
\end{lemma}

\begin{proof}
Let $g(x)$ be $g(x)=x^{2n}-x^{2n-i}-x^{i}-1$. Then $g(1)=-2 <0$ and $g(2)=2^{2n}-2^{2n-i}-2^{i}-1=(2^{2n-i}-1)(2^{i}-1)-2>0$.
$g^{\prime}(x)=(2 n) x^{2n-1}-(2n-i) x^{2n-i-1}-i x^{i-1}$. $g^{\prime}(1)=2n-(2n-i) -i=0$ and
$g^{\prime}(x)=x^{2 n-1} \left((2n) -(2n-i-1) x^{-i}-i x^{-2n+i}\right) \geq 0$ for $x \geq1$.
Hence the largest root of $g(x)$ lies between $1$ and $2$ and it is the only root bigger than $1$.
We can regard $g$ as a two variable polynomial $G(x,i)$.
$$\frac{\partial G}{\partial i}=x^i \ln{x} \left(x^{2 n-2 i}-1\right) \geq 0$$ if $x \geq 1$.
Hence, $$x^{2n}-x^{2n+1}-x^{2n-1}-1$$ gives the minimum spectral  cosidering Lemma~\ref{lemma:Perron}.
\end{proof}

\section{Proof of Theorem~\ref{thm:N} and Theorem~\ref{thm:S}}
Given a pseudo-Anosov homeomorphism $\psi$, let $\psi^{*}$ be the induced action of $\psi$ on the first cohomology and
$p_{\psi}(x)$ be the characteristic polynomial of $\psi^{*}$.

\subsection{Proof of Theorem~\ref{thm:N}}

Let $\Psi$ be the following set:

$$\begin{array}{ll}
 \Psi = &
\! \! \! \begin{Bmatrix}
 \psi \, : &\! \! \! \text{$\psi$  is a pseudo-Anosov homeomorphism of $N_{2k}$ with an orientable invariant }   \! \! \\
 &\! \! \!  \text{foliation and $p_{\psi}(x)$ has nonpositive coefficients except the first one \qquad \qquad    } \! \! \! \\
\end{Bmatrix}. 
\end{array}$$

Let $\Lambda$ be the following set: 
$$\Lambda=\{\lambda \, : \, \lambda \text{ is the dilatation of a pseudo-Anosov homeomorphism $\psi$, } \psi \in \Psi \}.$$ 

Let $\Gc$ be the following set: 

$$\begin{array}{ll}
 \Gc = &
 \begin{Bmatrix}
  p_{\psi}(x) \, : \, \psi \in \Psi\\
 \end{Bmatrix}. 
\end{array}$$

Note that $$x^{2k-1}-x^k-x^{k-1}-1$$ is in $\Gc$ by~\cite[Proposition 2.6]{LiechtiStrenner18}. 
 Since any polynomial in $\Gc$ can be considered as the characteristic polynomial of nonnegative Frobenius companion matrix which is also primitive,
the largest root of $$x^{2k-1}-x^k-x^{k-1}-1$$ is the minimal dilation among $\lambda$'s in $\Lambda$ 
by Lemma~\ref{lemma:mod2}  (1), Lemma~\ref{lemma:minispec}, and Lemma~\ref{lemma:Perron}.

\subsection{Proof of Theorem~\ref{thm:S}}
Let $\Psi$ be the following set:

$$\begin{array}{ll}
 \Psi = &
 \! \! \! \begin{Bmatrix}
  &\! \! \! \text{$\psi$  is a pseudo-Anosov homeomorphism of $S_{2k-1}$, which is orientation} \\
\psi \, : &\! \! \!  \text{ reversing with orientable  invariant foliations and either $p_{\psi}(x)$ or one}\\
  &\! \! \!  \text{ of $p(x) (x \pm 1)^2$, $p(x) (x^2 \pm 1)$, or $p(x) (x^2 \pm x + 1)$ has nonpositive \quad \ }\\
  &\! \! \!  \text{coefficients except the first one  \qquad \qquad \qquad \qquad \qquad \qquad \qquad \quad \quad }\\
 \end{Bmatrix}. 
\end{array}$$

Let $\Lambda$ be the following set: 
$$\Lambda=\{\lambda \, : \, \lambda \text{ is the dilatation of a pseudo-Anosov homeomorphism $\psi$, } \psi \in \Psi \}.$$ 

Let $\Gc$ be the following set: 

$$\begin{array}{lll}
 \Gc = &
\! \! \!  \begin{Bmatrix}
& \! \! \!  \text{g(x) has nonpositive coefficients except the first one and } g(x) \text{ is }\\
 g(x) \, : \,  &\! \! \!  \text{ one of } p_{\psi}(x),\ p_{\psi}(x) (x \pm 1)^2,\ p_{\psi}(x) (x^2 \pm 1), \text{ or } p_{\psi}(x) (x^2 \pm x + 1), \\
&\! \! \!  \psi \in \Psi  \qquad \qquad \qquad \qquad \qquad \qquad \qquad\qquad \qquad \qquad \qquad \quad\\
 \end{Bmatrix}. 
\end{array}$$

Note that $$x^{4k}-x^{2k+1}-x^{2k-1}-1$$ is in $\Gc$ by~\cite[Proposition 3.3]{LiechtiStrenner18}. Observe that
the largest positive root of $$x^{4k}-x^{2k+1}-x^{2k-1}-1$$ is less than the largest positive root of $$x^{4k-2}-x^{2k}-x^{2k-2}-1.$$
Since any polynomial in $\Gc$ can be considered as the characteristic polynomial of nonnegative Frobenius companion matrix which is also primitive, the largest positive root of
$$x^{4k}-x^{2k+1}-x^{2k-1}-1$$ is the minimal dilation among $\lambda$'s in $\Lambda$ by Lemma~\ref{lemma:mod2} (2) \& (3), Lemma~\ref{lemma:minispec}, and Lemma~\ref{lemma:Perron}.

\section{acknowledgement}
We thank Erwan Lanneau, Livio Liechti, Julien Marché, Alan Reid, Darren Long and anonymous referees. This work was supported by Basic Science Research Program through the National Research Foundation of Korea (NRF) funded by the Ministry of Education, Science and Technology (No. NRF-2018R1A2B6005847). The second author was supported by 2018 Hongik University Research Fund.

\providecommand{\bysame}{\leavevmode\hbox to3em{\hrulefill}\thinspace}
\providecommand{\MR}{\relax\ifhmode\unskip\space\fi MR }
\providecommand{\MRhref}[2]{%
  \href{http://www.ams.org/mathscinet-getitem?mr=#1}{#2}
}
\providecommand{\href}[2]{#2}

\end{document}